\documentclass[11pt]{article}
\usepackage[T1]{fontenc}
\usepackage[utf8]{inputenc}
\usepackage{lmodern}
\usepackage{amsmath,amssymb,amsthm}
\usepackage{microtype}
\usepackage{geometry}
\usepackage{enumitem}
\newtheorem{theorem}{Theorem}
\newtheorem{proposition}{Proposition}
\newtheorem{lemma}{Lemma}
\newtheorem{corollary}{Corollary}

\newcommand{\VCdim}{\operatorname{VCdim}}

\title{Polynomial-in-$r$ Bounds for Forbidden Traces of Uniform Hypergraphs}
\author{Pei Wu}
\date{\today}

\begin{document}

\maketitle

\begin{abstract}
We give a general principle that converts fixed-uniformity bounds for
forbidden traces into bounds with polynomial dependence on the
uniformity.  More precisely, let $H$ be a fixed set system on $h\ge1$
vertices, and suppose that, for some $\alpha\ge0$,
$\operatorname{ex}_j(m,\operatorname{Tr}(H))=O_{H,j}(m^\alpha)$ for
every fixed integer $j\ge1$.  Then, for every $\varepsilon>0$, there
is a constant $C_{H,\alpha,\varepsilon}$ such that
\[
  \operatorname{ex}_r(m,\operatorname{Tr}(H))
  \le C_{H,\alpha,\varepsilon} r^{h-1-\alpha+\varepsilon}m^\alpha
\]
for all $m\ge r\ge2$.  

In particular, for trace-$C_4$-free hypergraphs and every
$\varepsilon>0$ there is a constant $C_\varepsilon$ such that
\[
  \operatorname{ex}_r(m,\operatorname{Tr}(C_4))
  \le C_\varepsilon r^{3/2+\varepsilon}m^{3/2}
\]
for all $m\ge r\ge2$.  We also construct trace-$C_4$-free $r$-graphs
showing that 
\[
\operatorname{ex}_r(m,\operatorname{Tr}(C_4))
\ge c r^{1/2}m^{3/2}
\]
for an absolute constant $c>0$, for every fixed $r\ge3$ and all
sufficiently large $m$ (depending on $r$).
\end{abstract}

\section{Introduction}

\subsection{Traces and Tur\'an Problems}

A hypergraph $\mathcal H$ is a family of subsets of a finite ground
set.  The elements of the ground set are called the vertices of
$\mathcal H$, and the members of $\mathcal H$ are called its edges.  We
write $V(\mathcal H)$ and $E(\mathcal H)$ for its vertex set and edge
set, respectively.  If every edge has size $r$, then $\mathcal H$ is
said to be $r$-uniform, or simply an $r$-graph.

One of the basic questions in extremal combinatorics is the Tur\'an
problem.  Given a family $\mathcal F$ of $r$-graphs, one asks for the
largest possible number of edges in an $m$-vertex $r$-graph containing
no member of $\mathcal F$ as a subhypergraph.  This maximum is denoted
by
\[
  \operatorname{ex}_r(m,\mathcal F).
\]
When $r=2$ and $\mathcal F=\{F\}$ consists of a single graph, we use the
customary notation $\operatorname{ex}(m,F)$.

In this paper the forbidden objects occur as traces rather than as
ordinary subhypergraphs.  For a hypergraph $\mathcal H$ and a set
$Y\subseteq V(\mathcal H)$, the restriction of $\mathcal H$ to $Y$ is
\[
  \mathcal H|_Y:=\{H\cap Y:H\in E(\mathcal H)\},
\]
where repeated edges are discarded.  Let $F$ be another hypergraph.
We say that $\mathcal H$ contains $F$ as a trace if there is a set
$Y\subseteq V(\mathcal H)$ such that $\mathcal H|_Y$ contains a
subhypergraph isomorphic to $F$.  Equivalently, after identifying
$V(F)$ with a subset $Y$ of $V(\mathcal H)$, one can assign to every
edge $e\in E(F)$ a distinct edge $H_e\in E(\mathcal H)$ satisfying
\[
  H_e\cap Y=e.
\]
This notion is also known as an induced Berge copy of $F$; see, for
example, \cite{FurediLuo,LuoSpiro}.

For a fixed set system $F$, we write
\[
  \operatorname{ex}_r(m,\operatorname{Tr}(F))
\]
for the maximum number of edges in an $m$-vertex $r$-graph that does
not contain $F$ as a trace.  When $F$ is a graph and $r=2$, this
reduces to the ordinary Tur\'an number $\operatorname{ex}(m,F)$.

Our main theorem applies to any fixed forbidden trace $F$ for which
all fixed-uniformity extremal numbers admit upper bounds with a common
exponent of $m$.
Its proof combines a VC-dimension estimate with random restrictions;
the constants from the fixed-uniformity estimates may depend
arbitrarily on their uniformity.  We illustrate the result with
$F=C_4$.  F\"uredi and Luo \cite{FurediLuo} proved that if
$F$ is an outerplanar graph on $t$ vertices, then
\[
  \operatorname{ex}_r(m,\operatorname{Tr}(F))
  \le \frac12 r^r(t-2)^{r-2}\operatorname{ex}(m,F).
\]
This result is also recorded as Theorem~1.5 in \cite{LuoSpiro}.  In
particular,
\[
  \operatorname{ex}_r(m,\operatorname{Tr}(C_4))
  \le 2^{r-3}r^r\operatorname{ex}(m,C_4)
  =O\!\left(2^r r^r m^{3/2}\right).
\]
Thus the correct order $m^{3/2}$ was known for every fixed $r$, with
an explicit upper bound whose dependence on $r$ is
$\exp(O(r\log r))$.  Our result gives dependence
$r^{3/2+\varepsilon}$ for every $\varepsilon>0$.  Conversely, by
placing products of three chains over the triangles of a
$C_4$-free graph, we show that the asymptotic coefficient of
$m^{3/2}$ grows at least as $r^{1/2}$; see
Proposition~\ref{prop:c4-lower-bound}.  Here $m$ tends to infinity
with $r$ fixed before $r$ is allowed to grow.

\section{Preliminaries}

\subsection{Notation and VC Dimension}

We use $m$ for the size of the ambient ground set, $r$ for uniformity,
$s$ for the size of a random restriction, and $L$ for a fixed cutoff
on trace sizes.  We write $[m]=\{1,\ldots,m\}$ and use $M$ for the
ground-set size in the general VC-dimension estimates.  For a set system
$\mathcal F\subseteq2^{[m]}$, its $r$th level
is
\[
  \mathcal F_r:=\mathcal F\cap\binom{[m]}r.
\]
For nonnegative integers $a$ and $d$, we write
\[
  (a)_d:=a(a-1)\cdots(a-d+1)
\]
for the falling factorial, with $(a)_0=1$.

Recall that for $S\subseteq[m]$, the restriction of $\mathcal F$ to
$S$ is
\[
  \mathcal F|_S:=\{F\cap S:F\in\mathcal F\},
\]
where repetitions are discarded.  A set $S$ is \emph{shattered} by
$\mathcal F$ if $\mathcal F|_S=2^S$.  The VC dimension of $\mathcal F$,
denoted by $\VCdim(\mathcal F)$, is the maximum size of a set shattered
by $\mathcal F$.

The following observation supplies the VC-dimension bound used below.

\begin{proposition}\label{prop:vc-forbidden-trace}
Let $H$ be a set system on an $h$-element ground set, where $h\ge1$.  If
$\mathcal F$ is trace-$H$-free, then
\[
  \VCdim(\mathcal F)\le h-1.
\]
\end{proposition}

\begin{proof}
If an $h$-element set $S$ were shattered, every subset of $S$ would
occur as a trace of $\mathcal F$.  In particular, the traces would
contain a copy of $H$, contrary to the assumption.
\end{proof}

\subsection{Standard VC-Dimension Estimates}

The first tool is the Sauer--Shelah lemma, also known as the VC bound
\cite{Sauer,Shelah,VapnikChervonenkis}.

\begin{lemma}[Sauer--Shelah bound]\label{lem:sauer}
Let $\mathcal F$ be a set system on an $M$-element ground set.  If
$\VCdim(\mathcal F)\le d$, then
\[
  |\mathcal F|\le\sum_{i=0}^d\binom Mi.
\]
\end{lemma}

In particular, Proposition~\ref{prop:vc-forbidden-trace} and
Lemma~\ref{lem:sauer} imply that every trace-$H$-free set system on
an $M$-element ground set has size $O_H(M^{h-1})$.

We also require a standard estimate for the one-inclusion graph due to
Haussler, Littlestone, and Warmuth \cite{HausslerLittlestoneWarmuth}.  The
\emph{one-inclusion graph} $G_1(\mathcal F)$ of a set system
$\mathcal F$ is the graph
whose vertex set is $\mathcal F$, in which distinct members $A$ and
$B$ are adjacent when $|A\mathbin{\triangle}B|=1$.

\begin{lemma}[One-inclusion bound]\label{lem:one-inclusion}
If $\VCdim(\mathcal F)\le d$, then the one-inclusion graph of
$\mathcal F$ has at most $d|\mathcal F|$ edges.
\end{lemma}

The following consequence will be the form used in the proof of the
polynomial-in-$r$ bound.

\begin{lemma}[Random restriction lemma]\label{lem:random-restriction}
Let $\mathcal F$ be a set system on an $M$-element ground set with
$\VCdim(\mathcal F)\le d$.  If $S$ is chosen uniformly at random from
the $s$-subsets of the ground set, where $0\le d\le s\le M$, then
\[
  \mathbb E\!\left[\left|\mathcal F|_S\right|\right]
  \ge |\mathcal F|\frac{(s)_d}{(M)_d}.
\]
\end{lemma}

We prove Lemmas~\ref{lem:one-inclusion} and
\ref{lem:random-restriction} in Appendix~\ref{app:vc-proofs}.  If
$s\ge2d$, then each factor $(s-i)/s$, for $0\le i<d$, is at least
$1/2$.  Hence
\begin{equation}
  \frac{(s)_d}{(M)_d}
  \ge 2^{-d}\left(\frac{s}{M}\right)^d.
\label{eq:restriction-ratio}
\end{equation}

\section{A General Polynomial-in-$r$ Bound}

The hypothesis in the next theorem asks for one exponent $\alpha$
that works at every fixed uniformity.  The constants in those
fixed-uniformity estimates need not be controlled as $j$ grows.

\begin{theorem}\label{thm:polynomial-uniformity}
Let $H$ be a set system on $h\ge1$ vertices, and put $d=h-1$.
Let $\alpha\ge0$.  Assume that for every fixed integer $j\ge1$ there
is a constant $D_j$ such that
\begin{equation}\label{eq:fixed-uniformity-hypothesis}
  \operatorname{ex}_j(s,\operatorname{Tr}(H))
  \le D_j s^\alpha\qquad\text{for all }s\ge1.
\end{equation}
Then for every $\varepsilon>0$ there is a constant $C_{H,\alpha,\varepsilon}$
such that, for all integers $m\ge r\ge2$,
\[
  \operatorname{ex}_r(m,\operatorname{Tr}(H))
  \le C_{H,\alpha,\varepsilon}
       r^{d-\alpha+\varepsilon}m^\alpha.
\]
\end{theorem}

\begin{proof}
If $\alpha\ge d$, the Sauer--Shelah bound
(Lemma~\ref{lem:sauer}; see \cite{Sauer,Shelah,VapnikChervonenkis})
gives $\operatorname{ex}_r(m,\operatorname{Tr}(H))=O_H(m^d)$.
Since $m\ge r\ge2$,
\[
  m^d\le r^{d-\alpha+\varepsilon}m^\alpha,
\]
so the conclusion follows.  We therefore assume $0\le\alpha<d$.
Set $\delta=\varepsilon/(d-\alpha)>0$ and choose an integer $L$ with
\begin{equation}\label{eq:large-trace-cutoff}
  \delta L>d(1+\delta).
\end{equation}
Let $\mathcal G\subseteq\binom{[m]}r$ be trace-$H$-free, and write
$N=|\mathcal G|$.  Proposition~\ref{prop:vc-forbidden-trace} gives
$\VCdim(\mathcal G)\le d$.

First suppose $m>4d\,r^{1+\delta}$.  Choose uniformly at random
$S\in\binom{[m]}s$, where
\[
  s=\left\lfloor\frac{m}{r^{1+\delta}}\right\rfloor.
\]
Then $s\ge2d$ and $s/m\ge(2r^{1+\delta})^{-1}$.  If
$Z=\left|\mathcal G|_S\right|$ is the number of distinct traces, Lemma
\ref{lem:random-restriction} and
\eqref{eq:restriction-ratio} give
\begin{equation}\label{eq:general-trace-lower}
  \mathbb E Z
  \ge N\frac{(s)_d}{(m)_d}
  \ge 4^{-d}N r^{-d(1+\delta)}.
\end{equation}

Split $Z=Z_{<L}+Z_{\ge L}$ according to trace size.  Every level
of $\mathcal G|_S$ remains trace-$H$-free.  By
\eqref{eq:fixed-uniformity-hypothesis}, the levels of sizes
$1,\ldots,L-1$ contain at most $\sum_{j=1}^{L-1}D_j s^\alpha$
sets, and level zero contains at most one set.  Thus, for a constant
$K_L$ depending on the finitely many $D_j$,
\begin{equation}\label{eq:general-small-traces}
  Z_{<L}\le K_Ls^\alpha.
\end{equation}
For each $G\in\mathcal G$, if $L\le\min\{r,s\}$, then
\[
  \Pr(|G\cap S|\ge L)
  \le\mathbb E\binom{|G\cap S|}{L}
  =\binom rL\frac{(s)_L}{(m)_L}
  \le\frac1{L!}\left(\frac{rs}{m}\right)^L
  \le\frac{r^{-\delta L}}{L!}.
\]
If $L>\min\{r,s\}$, this probability is zero, so the same upper
bound holds.  Every distinct trace of size at least $L$ has such a preimage, so
\begin{equation}\label{eq:general-large-traces}
  \mathbb E Z_{\ge L}\le\frac{N}{L!}r^{-\delta L}.
\end{equation}
Combining \eqref{eq:general-trace-lower}--
\eqref{eq:general-large-traces} yields
\[
  4^{-d}N r^{-d(1+\delta)}
  \le K_Ls^\alpha+\frac{N}{L!}r^{-\delta L}.
\]
By \eqref{eq:large-trace-cutoff}, the final term can be absorbed
for all $r\ge r_0(H,\alpha,\varepsilon)$.  Since
$s\le m/r^{1+\delta}$, this gives
\begin{align*}
  N&\le 2\cdot4^dK_Lr^{d(1+\delta)}s^\alpha\\
   &\le 2\cdot4^dK_L
       r^{(d-\alpha)(1+\delta)}m^\alpha
    =2\cdot4^dK_L r^{d-\alpha+\varepsilon}m^\alpha.
\end{align*}
For the finitely many $2\le r<r_0$, the hypothesis
\eqref{eq:fixed-uniformity-hypothesis} gives the same conclusion
after enlarging the constant.

Finally, if $m\le4d\,r^{1+\delta}$, Lemma~\ref{lem:sauer} gives
\[
  N\le\sum_{i=0}^{d}\binom mi
   =O_H(m^d)
   =O_{H,\alpha,\varepsilon}
       \bigl(r^{(1+\delta)(d-\alpha)}m^\alpha\bigr).
\]
Because $(1+\delta)(d-\alpha)=d-\alpha+\varepsilon$, this is the
claimed bound in the remaining range.
\end{proof}

\subsection{The case of $C_4$}

Let $C_4$ be the cycle with edge set $\{ab,bc,cd,da\}$.
F\"uredi and Luo \cite{FurediLuo} proved that, for every fixed
$j\ge2$, there is a constant $D_j$ with
$\operatorname{ex}_j(s,\operatorname{Tr}(C_4))\le D_j s^{3/2}$
for all $s\ge1$.  The case $j=1$ follows from the trivial bound
$\operatorname{ex}_1(s,\operatorname{Tr}(C_4))\le s\le s^{3/2}$.
Taking $h=4$ and $\alpha=3/2$ in
Theorem~\ref{thm:polynomial-uniformity} gives the following.

\begin{corollary}\label{cor:c4-uniformity}
For every $\varepsilon>0$ there is a constant $C_\varepsilon$ such
that, for all integers $m\ge r\ge2$,
\[
  \operatorname{ex}_r(m,\operatorname{Tr}(C_4))
  \le C_\varepsilon r^{3/2+\varepsilon}m^{3/2}.
\]
\end{corollary}
\subsection{A Lower Bound on the Dependence on $r$}

We now construct trace-$C_4$-free hypergraphs showing that
\[
\liminf_{m\to\infty}
\frac{\operatorname{ex}_r(m,\operatorname{Tr}(C_4))}{m^{3/2}}
\]
tends to infinity as $r\to\infty$.

\begin{proposition}\label{prop:c4-lower-bound}
For every fixed integer $r\ge3$,
\[
  \liminf_{m\to\infty}
  \frac{\operatorname{ex}_r(m,\operatorname{Tr}(C_4))}{m^{3/2}}
  \ge \frac{(r-1)(r-2)}{12r^{3/2}}.
\]
\end{proposition}

\begin{proof}
We use the following facts:
\begin{enumerate}
  \item There is a $C_4$-free graph $G$ on $N$ vertices with
  $\left(\frac16+o(1)\right)N^{3/2}$ triangles as $N\to\infty$
  \cite[Theorem~4.9, with $t=2$]{AlonShikhelman}.
  \item On three disjoint ordered sets of size $r$, the family
  \[
    \bigl\{\{u_1,\ldots,u_a,v_1,\ldots,v_b,w_1,\ldots,w_c\}:
    1\le a,b,c\le r\bigr\}
  \]
  is trace-$C_4$-free.  Indeed, among any four vertices, two belong
  to the same ordered set, and their nested membership prevents a
  trace $C_4$.  This family is the product of three chains.
\end{enumerate}
Put $N=\lfloor m/r\rfloor$ and label $rN$ of the ambient vertices
by $(v,t)$, where $v\in[N]$ and $t\in[r]$.  The remaining vertices
belong to no hyperedge.  Choose $G$ as above with vertex set $[N]$,
and write $v_t=(v,t)$.
For each triangle $\{u,v,w\}$ of $G$, with $u<v<w$, include in
$\mathcal H$ all hyperedges of the form
$\{u_1,\dots,u_a,v_1,\dots,v_b,w_1,\dots,w_c\}$,
where $a,b,c\ge1$ and $a+b+c=r$.

To see that the construction is trace-$C_4$-free, suppose that four
vertices support a trace $C_4$.  No two of them can have the same
label $v$: if $i<j$, every hyperedge containing $(v,j)$ also contains
$(v,i)$, whereas a $C_4$ has an edge containing $(v,j)$ but not $(v,i)$.
Thus the four vertices have distinct labels, say $x_1,x_2,x_3,x_4$
in cyclic order.  Each consecutive pair of labels belongs to the
triangle defining a hyperedge that witnesses the corresponding
trace edge, so it is an edge of $G$.  These four pairs form a $C_4$
in $G$, a contradiction.

Finally, each triangle contributes exactly $\binom{r-1}{2}$
hyperedges, one for each positive integer solution of $a+b+c=r$.
Each hyperedge uniquely determines its supporting triangle and its
three prefix lengths, so no hyperedge is counted twice.  Writing $\tau(G)$
for the number of triangles in $G$, we obtain, for fixed $r$ as
$m\to\infty$,
\begin{align*}
  |\mathcal H|
  &=\binom{r-1}{2}\tau(G)\\
  &=\binom{r-1}{2}\left(\frac16+o(1)\right)
    \left\lfloor\frac mr\right\rfloor^{3/2}\\
  &=\left(\frac{(r-1)(r-2)}{12r^{3/2}}+o_r(1)\right)m^{3/2}.
\end{align*}
This proves the claimed lower bound.  Its coefficient is asymptotic
to $\sqrt r/12$ as $r\to\infty$, and hence is unbounded.
\end{proof}

\section{Concluding Remarks}

The restriction argument applies whenever the fixed-uniformity trace
bounds share an exponent $\alpha<h-1$; the case $\alpha\ge h-1$
follows directly from the Sauer--Shelah bound.  No control on the
dependence of the fixed-uniformity constants on the uniformity is
required.  In the trace-$C_4$ case, the argument
replaces the previously available $\exp(O(r\log r))$ coefficient
by $r^{3/2+\varepsilon}$ for all $m\ge r$.

Proposition~\ref{prop:c4-lower-bound} shows that a factor of order
at least $r^{1/2}$ is necessary even in the fixed-uniformity
asymptotic regime.  More precisely, if
\[
  a_r:=\limsup_{m\to\infty}
    \frac{\operatorname{ex}_r(m,\operatorname{Tr}(C_4))}{m^{3/2}},
\]
then for every $\varepsilon>0$ and all $r\ge3$,
\[
  \frac{(r-1)(r-2)}{12r^{3/2}}
  \le a_r\le C_\varepsilon r^{3/2+\varepsilon}.
\]
Determining the order of growth of $a_r$ remains open.

\section{AI-Assistance Disclosure}

The author used GPT-5.6 Sol to complete the proof and to assist with the exposition. The author independently reviewed and revised the outputs, verified the mathematical claims, and takes full responsibility for the final manuscript.
\appendix

\section{Proofs of the VC-Dimension Estimates}
\label{app:vc-proofs}

We prove the two VC-dimension estimates stated in the
Preliminaries.  The proof of the one-inclusion bound is included to keep
track of the constant needed in the restriction argument.

\begin{proof}[Proof of Lemma~\ref{lem:one-inclusion}]
We induct on the size $M$ of the ground set.  The assertion is
immediate when $M=0$ or $d=0$.  Fix a ground-set element $x$ and write
\[
  \mathcal F_0:=\{A\in\mathcal F:x\notin A\},
  \qquad
  \mathcal F_1:=\{A\setminus\{x\}:A\in\mathcal F,\ x\in A\}.
\]
Both families are regarded as set systems on the remaining $M-1$
elements.  Put
\[
  \mathcal U:=\mathcal F_0\cup\mathcal F_1,
  \qquad
  \mathcal I:=\mathcal F_0\cap\mathcal F_1.
\]
Then $|\mathcal F|=|\mathcal U|+|\mathcal I|$.  Restriction cannot
increase VC dimension, so $\VCdim(\mathcal U)\le d$.  Moreover,
$\VCdim(\mathcal I)\le d-1$: if $\mathcal I$ shattered a set $Y$ of
size $d$, then $\mathcal F$ would shatter $Y\cup\{x\}$.

There are exactly $|\mathcal I|$ one-inclusion edges in the $x$
direction.  For any other coordinate, an edge may occur in one or in
both of the two layers $\mathcal F_0$ and $\mathcal F_1$.  Its first
occurrence is represented by an edge of the one-inclusion graph of
$\mathcal U$; if it occurs in both layers, it is also an edge of the
one-inclusion graph of $\mathcal I$.  Consequently,
\[
  e(G_1(\mathcal F))
  \le e(G_1(\mathcal U))+e(G_1(\mathcal I))+|\mathcal I|.
\]
Applying the induction hypothesis to $\mathcal U$ and $\mathcal I$
gives
\[
  e(G_1(\mathcal F))
  \le d|\mathcal U|+(d-1)|\mathcal I|+|\mathcal I|
  =d|\mathcal F|,
\]
as required.
\end{proof}

\begin{proof}[Proof of Lemma~\ref{lem:random-restriction}]
Suppose first that a family $\mathcal G$ of VC dimension at most $d$
is supported on a $q$-element ground set $Q$, where $q\ge1$.
If a uniformly random ground-set element is
deleted, the number of distinct members of $\mathcal G$ that are lost
is the number of one-inclusion edges in the direction of that element.
Every one-inclusion edge has exactly one direction.  Lemma
\ref{lem:one-inclusion} therefore gives
\[
  \mathbb E\!\left[\left|\mathcal G|_{Q\setminus\{x\}}\right|\right]
  =|\mathcal G|-\frac{e(G_1(\mathcal G))}{q}
  \ge \left(1-\frac dq\right)|\mathcal G|,
\]
where $x$ is uniform in $Q$.

Starting with $\mathcal F$, delete elements successively and uniformly
until $s$ remain.  At every stage the restricted family still has VC
dimension at most $d$, and the final set is a uniformly random
$s$-subset.  Iterating the preceding inequality yields
\[
  \mathbb E\!\left[\left|\mathcal F|_S\right|\right]
  \ge |\mathcal F|\prod_{q=s+1}^{M}\left(1-\frac dq\right)
  =|\mathcal F|\frac{(s)_d}{(M)_d}.
\]
\end{proof}

\end{document}